\documentclass[12pt]{amsart}

\pdfoutput=1

\usepackage{mathtools, amssymb}
\usepackage{newtx}
\usepackage{enumerate}
\usepackage{bm, bbm}
\usepackage{tikz-cd}
\usepackage[nocompress]{cite}
\usepackage{hyperref}
\usepackage[a4paper, margin=2.71cm]{geometry}
\usepackage{microtype}

\theoremstyle{plain}
\newtheorem{theorem}{\bf Theorem}
\newtheorem{proposition}{\bf Proposition}
\newtheorem{lemma}{\bf Lemma}
\newtheorem{corollary}{\bf Corollary}
\theoremstyle{definition}
\newtheorem{definition}{\bf Definition}

\newtheorem{example}{\bf Example}

\newcommand		{\p}[1]	{\left(#1\right)}
\newcommand		{\pp}[1]	{\left[#1\right]}

\newcommand		{\abs}[1]{\left|#1\right|}
\newcommand 	{\norm}[1]{\left\lVert#1\right\rVert}

\renewcommand   {\P} {\mathbf P}
\newcommand    	{\G} {\mathbb G}
\renewcommand 	{\O} {\mathcal O}
\newcommand 	{\E} {\mathbf E}
\DeclareMathOperator 	{\var}{\mathbf{Var}}
\newcommand   	{\ber} {\mathrm{Ber}}

\newcommand    	{\1} {\mathbbm 1}

\newcommand 	{\er} {Erd\H{o}s-R\'enyi\ }
\newcommand 	{\ve} {\varepsilon}
\renewcommand 	{\t} {\mathbf t}
\newcommand 	{\aut} {\mathrm{Aut}}

\title[A Concentration Principle for Processes on Finite Graphs]
{A Random-Walk Concentration Principle\\ for Occupancy Processes on Finite Graphs}
\author{Davide Sclosa}
\author{Michel Mandjes}
\author{Christian Bick}
\date{\today}

\begin{document}

\begin{abstract}
This paper concerns discrete-time occupancy processes on a finite graph.
Our results can be formulated in two theorems, which are
stated for vertex processes, but also applied to edge process
(e.g., dynamic random graphs).
The first theorem shows that concentration of local state averages is controlled by a random walk on the graph.
The second theorem concerns concentration of polynomials of the vertex states.
For dynamic random graphs, this allows to estimate deviations of edge density, triangle density, and more general subgraph densities.
Our results only require Lipschitz continuity and hold for both dense and sparse graphs.
\end{abstract}

\maketitle


\section{Introduction}
\noindent
Occupancy processes form a broad class of stochastic processes on graphs,
covering various types of interacting particle systems and random graph models as special cases; cf.~\cite{cooper2013coalescing, hodgkinson2020normal, mcvinish2022dominating, avena2022discordant, donnelly1983finite, braunsteins2023sample, hanneke2010discrete}.
In this paper our specific focus is on analyzing these occupancy processes in terms of the finite-time deviations from the expected behavior, in addition quantifying how these deviations scale as the size of the graph grows.

Our main results have been formulated in two theorems. 
We state them for \emph{vertex processes} on a graph~$G$ where vertices turn on and off based on the state of the adjacent vertices. However, we also apply them to \emph{edge process} of a graph~$G$ where edges turn on and off---a dynamic random graph---by exploiting that this is the same as a vertex process on the line graph~$L(G)$.
The first theorem shows that concentration of local averages of states is controlled by a random walk on the graph.
The second theorem concerns concentration of polynomial observables of the states: 
For dynamic random graphs, this allows to estimate triangle counts and more general subgraph densities.
In contrast to the analytical approach in~\cite{hodgkinson2020normal}, which requires the existence of third order derivatives, our approach is of a combinatorial nature and only requires Lipschitz continuity. 
Our results are further related to work in~\cite{braunsteins2023sample, crane2016dynamic, crane2017exchangeable, vcerny2020markovian}, but 
apply to arbitrary finite graphs without requiring the theory of graph limits.

\smallskip

Consider a finite graph~$G=(V,E)$. For any vertex $v\in V$, we denote by~$N_v \subseteq V$ the set of its neighbors.
For~$X\in \{0,1\}^V$ we call~$X_{N_v} = (\deg v)^{-1} \sum_{w\in N_v} X_w$ a \emph{neighborhood average}.
Throughout this paper we denote, for two random variables~$X$ and~$Y$, by~$\E^Y (X)$ the conditional expectation of~$X$ given~$Y$.

\begin{definition} 
\label{def:vertex_process}
Fix a finite graph~$G=(V,E)$ and two functions~$f,g: [0,1]\to[0,1]$.
An \emph{occupancy process} is a discrete-time
Markov process~$(X(t))_{t\geq 0} \subseteq \{0,1\}^V$ with
\begin{enumerate} [(i)]
\item
	\emph{neighbor interaction}:
	Each vertex~$v\in V$ switches from~$1$ to~$0$
	with probability~$f(X_{N_v})$ and from~$0$ to~$1$
	with probability~$g(X_{N_v})$;
\item
	\emph{independent updates}:
	The random variables~$(X_v(t+1))_{v\in V}$ are conditionally joint independent
	given~$X(t)$. \label{ax:vertices_2}
\end{enumerate}
\end{definition}

The occupancy process~$X(t)$ can be described as a Markov chain on~$\{0,1\}^V$
by a linear map~$[0,1]^{\{0,1\}^V}\to [0,1]^{\{0,1\}^V}$.
Condition~(i) in Definition~\ref{def:vertex_process} can be written compactly as
\begin{equation}
\label{eq:stoch_process}
\E^{X(t)} (X_v(t+1)) = \gamma(X_v(t), X_{N_v}(t))    
\end{equation}
where~$\gamma(x,y) = x(1-f(y)) + (1-x)g(y)$ for~$x,y\in [0,1]$.
This suggests to consider the following deterministic process described by a (possibly non-linear) map~$[0,1]^V\to [0,1]^V$.

\begin{definition}
Fix a finite graph~$G=(V,E)$ and two functions~$f,g: [0,1]\to[0,1]$.
Write~$\Gamma: [0,1]^V \to [0,1]^V$ for the map with components
	$\Gamma(x)_v = \gamma(x_v,x_{N_v})$.
The \emph{deterministic process} with initial condition $x(0)\in [0,1]^V$ corresponds to the iteration~$x(t+1) = \Gamma (x(t))$ with
\begin{equation}
\label{eq:det_process}
    x_v(t+1) = \gamma(x_v(t), x_{N_v}(t)).
\end{equation}
\end{definition}


\subsection{Main results}
Our theorems can be seen as dynamical counterparts of the (static) concentration of measure
principle, which (informally) states that a function of many weakly dependent random variables that is not too
sensitive to any individual variable, is concentrated around its expected 
value~\cite{kontorovich2017concentration, talagrand1996new, ledoux2001concentration}.

Our first theorem concerns concentration of neighborhood averages~$X_{N_v}(t)$.
It implies that if the~$(X_v(0))_{v\in V}$ are independent
and vertex degrees are large, then~$X_{N_v}(t)$ is concentrated
around its expected value.
Specifically, it entails that~$\var(X_{N_v}(t)) \leq \ve$
for~$0\leq t\leq \O(\log (\ve \delta(G)))$, where~$\delta(G)$ is the minimum degree in the graph~$G$.

The particular case~$1-f(x)=g(x)=x$ for every~$x\in [0,1]$,
known as the \emph{voter model},
is effectively analyzed by looking at a dual process, a coalescent random
walk on the graph~$G$~\cite{avena2022discordant, donnelly1983finite,
durrett1994particle, cooper2013coalescing, hassin2001distributed}.
For general non-linear~$f$ and~$g$ such duality does not hold.
However, we prove that in this case a random walk on~$G$ does control deviations.

\begin{theorem} [Concentration of Neighborhood Averages] \label{thm:main}
Suppose that~$f$ and~$g$ are $M$-Lipschitz.
Fix~$x(0)\in \{0,1\}^V$ and let~$\P(X(0)=x(0))=1$.
Then, for every vertex~$v\in V$ and every~$t\geq 0$,
\begin{equation} \label{eq:main}
	\E\abs{X_{N_{v}}(t) - x_{N_{v}}(t)}
		\leq \sum_{s=0}^{t-1} {t \choose {s+1}}
			\pp{\sum_{w\in V}
        \frac{\pi_{v,w}(s)}{\sqrt{\deg(w)}}} M^s,
\end{equation}
where~$\pi_{v,w}(s)$ is the probability that a random walker on~$G$
starting at the vertex~$v \in V$
occupies the vertex~$w\in V$ after~$s$ steps.
\end{theorem}

The quantity between square brackets in~\eqref{eq:main}
is the expected reciprocal square-root degree seen by a random walker
that has started at~$v$ and moved~$s$ steps.
Such quantity is uniformly bounded by~$\delta(G)^{-1/2}$,
where~$\delta(G)$ is the minimum degree in the graph~$G$.
Therefore,
\begin{equation} \label{eq:main_simple}
		\E\abs{X_{N_{v}}(t) - x_{N_{v}}(t)} \leq C \delta(G)^{-1/2},
\end{equation}
where~$C = M^{-1} \p{(M+1)^t-1}$ depends only on~$f$, $g$, and~$t$.
In other words, for fixed~$t$, as~$\delta(G)\to \infty$
the random variable~$X_{N_{v}}(t)$ converges (e.g., concentrates) in probability
to the constant~$x_{N_{v}}(t)$.
Note that~\eqref{eq:main_simple} is a significant weakening of~\eqref{eq:main},
since~\eqref{eq:main_simple} requires all vertices to have large degree,
while~\eqref{eq:main} only
requires frequently visited vertices to have large degree.
Moreover, note that these bounds depend on degrees
but not on the number of vertices: 
The result holds for both dense and sparse graphs.
Finally, note that the exponent in~$\delta(G)^{-1/2}$ is the best possible: 
If~$G$ is regular and~$(X_v(t))_{v\in V}$ are i.i.d.,
then~$\E\abs{X_{N_{v}}(t)- \E X_{N_{v}}(t)} \propto \delta(G)^{-\frac{1}{2}}$,
which follows by applying the normal approximation of the binomial distribution.

Our second main theorem gives a concentration result for \emph{polynomial observables}, that is, general homogeneous polynomials in the state variables~$X$.
Note that neighborhood averages~$X_{N_v} = (\deg v)^{-1} \sum_{w\in N_v} X_w$ are particular homogeneous polynomials of degree~$1$.
The extension to general polynomial observables, allows to estimate triangle counts and general homomorphism densities for the evolution of dynamic random graphs.
The result will be stated for a class of polynomials.

\begin{definition} \label{def:combinatorial_norms}
Fix an integer~$d\geq 1$. 
Consider a $d$-homogeneous polynomial with real coefficients given by~$P(X) = \sum_{W \in {V \choose d}} P_W X_{w_1} X_{w_2} \dotsb X_{w_d}$ with $W=\{w_1, \dotsc, w_d\}$.
For a vertex~$v\in V$ set $c_v(P) = \sum_{W: v\in W \in {V\choose d}} \abs{P_W}$. Define
\begin{align*}
		 \norm{P}_p &= \p{\sum_{v\in V} c_v(P)^p}^{1/p}.    
\end{align*}
Let~$\mathcal P_d(\lambda,\rho)$ be the set of polynomials~$P$ of this form and
such that~$\norm{P}_1\leq \lambda$ and~$\norm{P}_2\leq \rho$.
\end{definition}

To avoid cumbersome formulas, we state the following
result using $\O$-notation.

\begin{theorem} [Concentration of Polynomial Observables] \label{thm:poly}
Suppose that~$f$ and~$g$ are $M$-Lipschitz.
Fix~$x(0)\in \{0,1\}^V$ and let~$\P(X(0)=x(0))=1$. Then
\[
	\sup_{P\in \mathcal P_d(\lambda,\rho)} \E\abs{P(X(t))-P(x(t))}
		= \O_{d,t}(\rho (1+M\lambda)^t),
\]
which means that there is a constant~$C = C_{d,t}$
independent on~$G$, $f$, $g$, $X(0)$, $\lambda$, and~$\rho$,
such that the left hand side is bounded by~$C \rho (1+M\lambda)^t$.
\end{theorem}

Note that Theorem~\ref{thm:main} concerns concentration of linear polynomials with combinatorial meaning (neighborhood averages),
and gives an explicit upper bound which also has combinatorial meaning (in terms of random walks).
By contrast, Theorem~\ref{thm:poly} concerns concentration of general polynomial observables and gives an algebraic upper bound in terms of their norms.
While Theorem~\ref{thm:main} bounds deviations
of a single observable, Theorem~\ref{thm:poly}
gives a uniform bound over an infinite set of observables.
Theorem~\ref{thm:main} and Theorem~\ref{thm:poly}
are proven in Section~\ref{sec:proof} and Section~\ref{sec:proof_poly}, respectively.

\subsection{Consequences: deterministic methods and dynamic random graphs} 
We give applications of our main theorems in two areas in Sections~\ref{sec:dynamics} and~\ref{sec:random_graphs}, respectively.

In Section~\ref{sec:dynamics} we propose a deterministic method that quantifies the difference between~$x(t)$
and~$\E X(t)$, and analyzes the dynamics of~$x(t)$ in order to obtain information
about the dynamics of~$X(t)$.
This deterministic method applies to both vertex processes and edge processes.
The vertex processes in this paper---while closely related to classical interactive particle systems such as Ising model~\cite{holley1974recent, holley1986logarithmic},
percolation models~\cite{shante1971introduction}, as well as epidemic (such as susceptible-infected-susceptible or SIS),
contact, and voter models~\cite{liggett1985interacting, liggett1999stochastic, pare2018analysis, rodriguez2018use, ahn2014mixing}---differ in the following aspects: time is discrete rather than continuous, interaction functions are non-linear, and graphs are finite and arbitrary.

Then, in Section~\ref{sec:random_graphs}, we apply our results to (dynamic) random graphs. We prove concentration for homomorphism densities and show how simple rules can lead to chaotic random graphs. The idea of analyzing random graphs by a suitable deterministic approximation has appeared in work by N.~C.~Wormald~\cite{wormald1995differential}, with the difference that in our paper observables grow in number with the number of vertices (e.g., neighborhood averages in Theorem~\ref{thm:main}) or are infinite in number (polynomials in Theorem~\ref{thm:poly}).
In contrast to~\cite{braunsteins2023sample, crane2016dynamic, crane2017exchangeable, vcerny2020markovian}, and similar to~\cite{hodgkinson2020normal}, the approach in our paper is finitary: Asymptotics are obtained directly on finite graphs, without introducing graph limits. In particular, our results apply to both sparse and dense graph sequences, without requiring any property of convergence.


\section{Proof of Theorem~\ref{thm:main}} \label{sec:proof}
\noindent
The proof of Theorem~\ref{thm:main} is divided into two parts:
The first, ``probabilistic'' part obtains
an evolution equation
for~$\zeta_v(t) = \E\abs{X_{N_v}(t) - x_{N_v}(t)}$.
The second, ``deterministic'' part obtains a useful expression for the solutions of such equation.

The difficulty lies in the fact that~$\zeta_v(t+1)$ cannot be expressed as a function of~$\zeta_v(t)$. In fact, $\zeta_v(t+1)$~cannot even be expressed
as function of the collection~$(\zeta_w(t))_{w\in V}$.
Indeed, we will observe that~$\zeta_v(t+1)$ depends on~$(\zeta_w(t))_{w\in V}$,
as well as on some quantity~$\zeta_v^{x(t)}(t) = \E\abs{P^{x(t)}(X(t))-P^{x(t)}(x(t))}$,
where~$X\mapsto P^{x(t)}(X)$ is a linear map that is not too far from the
linear map~$X\mapsto X_{N_v}$.
To overcome this issue, we prove concentrations for general linear functionals.

The following lemma constitutes the probabilistic part of the proof.

\begin{lemma} \label{lem:proof_1}
Consider a linear map~$P(X)=\sum_{v\in V} P_{v} X_v$. Under the hypothesis
of Theorem~\ref{thm:main}, the following inequality holds:
\begin{equation}
\begin{split}
	& \E\abs{P(X(t+1)) - P(x(t+1))} \\
	& \qquad\leq
		\norm{P}_2
		+ \E\abs{P^{x(t)}(X(t))-P^{x(t)}(x(t))}
		+ M \sum_{v\in V} \abs{P_v} \E \abs{X_{N_v}(t) - x_{N_v}(t)} \label{eq:proof_1}
\end{split}
\end{equation}
where~$P^{x(t)}(X)=\sum_{v\in V} P_{v}^{x(t)} X_v$
has coefficients~$P^{x(t)}_v = P_v (1-f(x_{N_v}(t))-g(x_{N_v}(t)))$.
\end{lemma}
\begin{proof}
For the stochastic quantities set~$X=X(t)$ and~$X'=X(t+1)$ and write~$F_{N_v} = f(X_{N_v})$ and~$G_{N_v} = g(X_{N_v})$.
Analogously, set~$x=x(t)$ and~$x'=x(t+1)$ for the deterministic evolution and write~$f_{N_v} = f(x_{N_v})$ and~$g_{N_v} = g(x_{N_v})$.
Then,
\begin{align*}
	\E^X X'_v &= (1-F_{N_v}-G_{N_v}) X_v + G_{N_v} \\
	x'_v &= (1-f_{N_v}-g_{N_v}) x_v + g_{N_v}.
\end{align*}
Therefore,
\begin{subequations}
\begin{align}
	 \E^X P(X') - P(x')
	&= \sum_{v\in V}
			P_v (1-f_{N_v}-g_{N_v}) (X_v-x_v) \label{eq:pr1} \\
		& \qquad + \sum_{v\in V}
			P_v \pp{(f_{N_v}-F_{N_v}+g_{N_v}-G_{N_v}) X_v + (G_{N_v}-g_{N_v})}.
			\label{eq:pr2}
\end{align}
\end{subequations}
Define~$P^x_v = P_v (1-f_{N_v}-g_{N_v})$ and note that~$P^x_v$ is a deterministic quantity.
Then the right hand side of~\eqref{eq:pr1} is equal to~$P^x(X)-P^x(x)$.
Since~$X_v\in \{0,1\}$, for every~$v\in V$ the quantity square brackets in~\eqref{eq:pr2}
is bounded by~$M\abs{X_{N_v} - x_{N_v}}$. It follows that
\[
	\E \abs{\E^X P(X') - P(x')} \leq \E \abs{P^x(X)-P^x(x)}
		+ M \sum_{v\in V} \abs{P_v} \E \abs{X_{N_v} - x_{N_v}}.
\]
On the other hand, since~$(X'_v)_{v\in V}$ are independent given~$X$,
we have~$\var^X P(X') \leq \sum_{v\in V} \abs{P_v}^2$. Thus, by Jensen's
inequality,
\[\E^X \abs{P(X') - \E^X P(X')} \leq \p{\var^X P(X')}^{1/2} \leq \norm{P}_2.\]
Therefore,
\begin{equation} \label{eq:pr3}
	\E \abs{P(X') - \E^X P(X')}
	= \E \p{\E^X \abs{P(X') - \E^X P(X')}}
	\leq \E \p{\norm{P}_2}
	= \norm{P}_2.
\end{equation}
The statement follows from~\eqref{eq:pr1},~\eqref{eq:pr3}, and the triangle inequality.
\end{proof}

Lemma~\ref{lem:proof_1} expresses~$\E\abs{P(X(t+1)) - P(x(t+1))}$
in terms of the quantities
\[
	(\E \abs{X_{N_v}(t) - x_{N_v}(t)})_{v\in V},\ \E\abs{P^{x(t)}(X(t))-P^{x(t)}(x(t))}.
\]
Before substituting~$P(X)=X_{N_v}$, which is our case of interest,
we need to eliminate the spurious quantities~$P^{x(t)}(X(t))$.
This is done in the following lemma.

\begin{lemma} \label{lem:proof_2}
Under the hypothesis of Lemma~\ref{lem:proof_1}, the following inequality holds:
\begin{equation} \label{eq:proof_2}
	\E\abs{P(X(t)) - P(x(t))} \leq
		t \norm{P}_2
		+ M \sum_{s=1}^{t-1} \sum_{v\in V} \abs{P_v} \E \abs{X_{N_v}(s) - x_{N_v}(s)}.
\end{equation}
\end{lemma}
\begin{proof}
We are going to prove by induction that for every~$P$ and every~$t\geq 0$
the inequality~\eqref{eq:proof_2} holds.
The statement is trivial for~$t=0$.
Fix~$t>0$ and assume the statement true for~$t$ and all~$P$.
By Lemma~\ref{lem:proof_1},
\begin{align*}
	& \E\abs{P(X(t+1)) - P(x(t+1))} \\
	& \leq \norm{P}_2
		+ \E\abs{P^{x(t)}(X(t))-P^{x(t)}(x(t))}
		+ M \sum_{v\in V} \abs{P_v} \E \abs{X_{N_v}(t) - x_{N_v}(t)} \\
	& \leq \norm{P}_2
		+ t\norm{P^{x(t)}}_2
			+ M \sum_{s=1}^{t-1} \sum_{v\in V} \abs{P^{x(t)}_v} \E \abs{X_{N_v}(s) - x_{N_v}(s)} \\
	    & \quad
        + M \sum_v \abs{P_v} \E \abs{X_{N_v}(t) - x_{N_v}(t)} \\
	& \leq (t+1)\norm{P}_2
			+ M \sum_{s=1}^{t} \sum_k \abs{P_v} \E \abs{X_{N_v}(s) - x_{N_v}(s)},
\end{align*}
where the last inequality follows from~$\abs{P^{x(t)}_v} \leq \abs{P_v}$
and~$\norm{P^{x(t)}}_2 \leq \norm{P}_2$.
\end{proof}

We are now ready to complete the proof of Theorem~\ref{thm:main}.

\begin{proof} [Proof of Theorem~\ref{thm:main}]
First, note that if~$P(X)=X_{N_{v}}$
for some vertex~$v\in V$,
then
\[
    \abs{P_{w}} =
    \begin{dcases*}
        \deg(v)^{-1} & if $w\in N_v$, \\
        0 & otherwise,
    \end{dcases*}
\]
and, in particular,~$\norm{P}_2 = \deg({v})^{-1/2}$.

Second, define~$\zeta_v(t) = \E\abs{X_{N_v}(t) - x_{N_v}(t)}$
for all~$v\in V$.
Then by Lemma~\ref{lem:proof_2} it follows that for every vertex~$v$ and every~$t\geq 0$,
\begin{equation} \label{eq:bound1_app}
	\zeta_v(t) \leq
		t \deg(v)^{-1/2}
		+ M \sum_{s=1}^{t-1} \sum_{w\in N_v} \frac{1}{\deg v} \zeta_w(s).
\end{equation}

We now prove by induction that
for every vertex~$v_1$ and every~$t\geq 0$ we have
\begin{equation} \label{eq:bound2_app}
	\zeta_{v_1}(t)
		\leq \sum_{s=1}^{t} {t \choose s} M^{s-1}
		\pp{
		\sum_{v_2 \in N_{v_1}} \cdots \sum_{v_{s} \in N_{v_{s-1}}}
		\frac{\deg(v_s)^{-1/2}}{\deg({v_1})\cdots \deg({v_{s-1}})}
		},
\end{equation}
with the conventions that~$0^0=1$ and that for~$s=1$
the quantity between square brackets is equal to~$\deg(v_1)^{-1/2}$.

For~$t=0$ the inequality~\eqref{eq:bound2_app} is trivial.
For~$t=1$ the inequality~\eqref{eq:bound2_app} holds
by~\eqref{eq:bound1_app} and by convention.
Fix~$t>1$ and assume~\eqref{eq:bound2_app} true
for~$t$ and for all~$v_1\in V$.
Then we have
\allowdisplaybreaks
\begin{align*}
	& \zeta_{v_1}(t+1) \\
	& \leq 	(t+1) \deg(v_1)^{-1/2}
		+ M \sum_{s=1}^{t} \sum_{v_2 \in N_{v_1}} \frac{1}{\deg v_1} \zeta_{v_2}(s) \\
	& \leq (t+1) \deg(v_1)^{-1/2}\\
		& \quad + \sum_{s=1}^{t} \sum_{r=1}^{s} {s\choose r} M^{r}
		\pp{
		\sum_{v_2 \in N_{v_1}} \sum_{v_3 \in N_{v_2}} \cdots \sum_{v_{r+1} \in N_{v_{r}}}
		\frac{\deg(v_{r+1})^{-1/2}}{\deg({v_1})\cdots \deg({v_{r}})}
		} \\
	& = (t+1) \deg(v_1)^{-1/2}\\
		& \quad + \sum_{r=2}^{t+1} \p{\sum_{s=r-1}^{t} {s \choose {r-1}}} M^{r-1}
		\pp{
		\sum_{v_2 \in N_{v_1}} \sum_{v_3 \in N_{v_2}} \cdots \sum_{v_{r} \in N_{v_{r-1}}}
		\frac{\deg(v_{r})^{-1/2}}{\deg({v_1})\cdots \deg({v_{r-1}})}
		} \\
	& = (t+1) \deg(v_1)^{-1/2}\\
		& \quad + \sum_{r=2}^{t+1} {{t+1} \choose r} M^{r-1}
		\pp{
		\sum_{v_2 \in N_{v_1}} \sum_{v_3 \in N_{v_2}} \cdots \sum_{v_{r} \in N_{v_{r-1}}}
		\frac{\deg(v_{r})^{-1/2}}{\deg({v_1})\cdots \deg({v_{r-1}})}
		}.
\end{align*}
Since for~$r=1$ we have~${t+1 \choose r} M^{r-1} = (t+1)$,
this proves~\eqref{eq:bound2_app}.
Theorem~\ref{thm:main} follows
from~\eqref{eq:bound2_app} and the definition of~$\pi_{w,v}$
given in the statement.
\end{proof}


\section{Proof of Theorem~\ref{thm:poly}} \label{sec:proof_poly}
\noindent
Similar to the proof of Theorem~\ref{thm:main}, the proof of Theorem~\ref{thm:poly} is divided into two parts.
The first, probabilistic part obtains an evolution equation for~$\E\abs{P(X(t)) - P(x(t))}$.
The second, deterministic part obtains a useful expression for the solutions of such equation.

As in Theorem~\ref{thm:main},
the difficulty of the proof relies on the fact that the expected
error~$\E\abs{P(X(t+1)) - P(x(t+1))}$ cannot be expressed as a function of~$\E\abs{P(X(t)) - P(x(t))}$.
Indeed, we will observe that the expected error at time~$t+1$
of a polynomial of degree~$d$
depends on the expected errors at time~$t$ of certain polynomials
with degrees up to~$d$. These polynomials
depend on both~$P$ and the deterministic trajectory.
The bulk of the proof is controlling the
number and the norms of such polynomials.

The following lemma constitutes the probabilistic part of the proof and is analogous to Lemma~\ref{lem:proof_1} above.

\begin{lemma} \label{lem:proof_2_1}
Consider an homogeneous polynomial of degree~$d$,
\[
	P(X) = \sum_{W \in {V \choose d}} P_W X_{w_1} X_{w_2} \cdots X_{w_d}.
\]
Under the hypothesis
of Theorem~\ref{thm:poly}, the following inequality holds:
\begin{equation}
\begin{split}
	\E\abs{P(X(t+1))-P(x(t+1))}
	& \leq 2 \norm{P}_2 + \sum_{1\leq d' \leq d} \E\abs{P^{d',x(t)}(X(t)) - P^{d',x(t)}(x(t))} \\
	& \quad\qquad + 2^d M \sum_{v\in V} c_v(P) \E \abs{X_{N_v}(t)-x_{N_v}(t)},
  \label{eq:proof_2_1}
\end{split}
\end{equation}
where the polynomials~$P^{d',x}$ will be defined in the proof.
\end{lemma}

\begin{proof}
While in Lemma~\ref{lem:proof_1} we used
Bienaym\'e's identity
and Jensen's inequality, here we use
McDiarmid's inequality.
Recall from
Definition~\ref{def:combinatorial_norms} that~$c_v(P) = \sum_{W: v\in W \in {V\choose d}} \abs{P_W}$
for a vertex~$v\in V$. Then McDiarmid's inequality implies
\begin{equation} \label{eq:proof_poly_P}
	\E^X \abs{P(X')-\E^X P(X')} \leq 2
    \p{\sum_{v\in V} c_v(P)^2}^{\frac{1}{2}} = 2 \norm{P}_2,
\end{equation}
with~$\norm{\cdot}_2$ defined as in Definition~\ref{def:combinatorial_norms}.
Since the random variables~$(X'_v)_{v\in V}$ are
independent given~$X$, we have~$\E^X(P(X')) = P(\E^X(X'))$. 
Hence,
\begin{equation}\label{eq:proof_poly_start}
	\E^X(P(X')) - P(x')
	= \sum_{W \in {V \choose d}} P_W \p{\prod_{w\in W} \E^X(X'_{w}) - \prod_{w\in W} x'_{w}}.
\end{equation}

We now introduce some shorthand notation.
For the quantities that depend on the stochastic evolution~$X$, write $F_{N_w} = f(X_{N_w})$, $G_{N_w} = g(X_{N_w})$ and define
\begin{equation}
\label{eq:ShorthandPhi}
	\prod_{w\in W} \E^X(X'_{w})
	= \sum_{U\subseteq W}
		\underbrace{\p{\prod_{w\in U}(1-F_{N_w}-G_{N_w})
			\prod_{w\in W\setminus U} G_{N_w}}}_{:=\phi_{W,U}^{F_W,G_W}}
		\prod_{w\in U} X_w
\end{equation}
For the quantities that depend on the deterministic evolution~$x$, define the analogous deterministic quantities~$f_{N_w}=f(x_{N_w})$, $g_{N_w}=g(x_{N_w})$, and finally~$\phi_{W,U}^{f_W,g_W}$ from the expansion
of~$\prod_{w\in W} x'_{w}$ corresponding to~\eqref{eq:ShorthandPhi}.

It follows that
\begin{align*}
	& \prod_{w\in W} \E^X(X'_{w}) - \prod_{w\in W} x'_{w} \\
	& = \sum_{U\subseteq W}
		\p{\phi_{W,U}^{F_W,G_W} \prod_{w\in U} X_w - \phi_{W,U}^{f_W,g_W} \prod_{w\in U} x_w} \\
	& = \sum_{U\subseteq W}
		\p{
		\phi_{W,U}^{f_W,g_W} \p{\prod_{w\in U} X_w - \prod_{w\in U} x_w}
		+ \prod_{w\in U} X_w \p{\phi_{W,U}^{F_W,G_W} - \phi_{W,U}^{f_W,g_W}}
		} \\
	& = \p{
		\sum_{\emptyset \neq U \subseteq W}
		\phi_{W,U}^{f_W,g_W} \p{\prod_{w\in U} X_w - \prod_{w\in U} x_w}
		}
		+ \Psi_{W}(X,x),
\end{align*}
where~$\Psi_{W}(X,x)$ is sum of~$2^{\abs{W}}=2^d$ terms where each term is a product
of some~$X_w$ with~$w\in W$
and some~$\phi_{W,U}^{F_W,G_W} - \phi_{W,U}^{f_W,g_W}$ with~$U\subseteq W$.
Substitution into~\eqref{eq:proof_poly_start} gives
\begin{align*}
	\E^X [P(X')] - P(x')
	& =
	\sum_{1\leq d' \leq d}
		\sum_{U \in {V\choose d'}}
		\p{\sum_{W: U\subseteq W \in {V\choose d}} P_W \phi_{W,U}^{f_W,g_W}}
			\p{\prod_{w\in U} X_w - \prod_{w\in U} x_w} \\
		& \quad + \sum_{W\in {V\choose d}} P_W \Psi_W(X,x).
\end{align*}
Recall that the quantities~$\phi_{W,U}^{f_W,g_W}$ are deterministic:
They depend on~$x$ and not on~$X$. Write
\[
	\E^X [P(X')] - P(x')
	=
	\sum_{1\leq d' \leq d}
		\p{P^{d',x}(X) - P^{d',x}(x)}
		+ \sum_{W\subseteq {V\choose d}} P_W \Psi_W(X,x),
\]
where
\[
	P^{d', x}(X) = \sum_{U\in {V\choose d'}}
		\p{\sum_{W: U\subseteq W \in {V\choose d}} P_W \phi_{W,U}^{f_W,g_W}}
			\prod_{w\in U} X_w
			:= \sum_{U\in {V\choose d'}} P^{d', x}_{U} \prod_{w\in U} X_w.
\]
Note that~$\abs{\phi_{W,U}^{f_W,g_W}}\leq 1$. Hence,
\begin{equation} \label{eq:dd'-1}
	\norm{P^{d', x}}_2^2 = \sum_{v\in V} \abs{\sum_{U:v\in U\in {V\choose d'}} P^{d', x}_{U}}^2
		\leq \sum_{v\in V} \p{\sum_{U: v\in U\in {V\choose d'}}
			\sum_{W: U \subseteq W \in {V\choose d}} \abs{P_W}}^2
			= {d-1\choose d'-1}^2 \norm{P}_2^2.
\end{equation}
On the other hand,
\[
	\abs{\Psi_W(X,x)}
	\leq \sum_{U\subseteq W} \abs{\phi_{W,U}^{F_W,G_W} - \phi_{W,U}^{f_W,g_W}}
	\leq \sum_{U\subseteq W} \sum_{w\in U} 2M \abs{X_{N_w}-x_{N_w}}
	= 2^d M \sum_{w\in W} \abs{X_{N_w}-x_{N_w}}
\]
from which we obtain
\[
	\abs{\sum_{W \in {V\choose d}} P_W \Psi_W(X,x)}
	\leq 2^d M \sum_{W \in {V\choose d}} \abs{P_W} \sum_{w\in W} \abs{X_{N_w}-x_{N_w}}
	= 2^d M \sum_{v\in V} c_v(P) \abs{X_{N_v}-x_{N_v}}.
\]
Combining these estimates with~\eqref{eq:proof_poly_P} and triangle inequality, we obtain~\eqref{eq:proof_2_1} as claimed.
\end{proof}

This completes the probabilistic part of the argument.
We now conclude the proof by presenting
the deterministic part of the argument.

\begin{proof} [Proof of Theorem~\ref{thm:poly}]
Similar to Lemma~\ref{lem:proof_2} in the proof of Theorem~\ref{thm:main},
we first want to eliminate the spurious terms~$P^{d',x}$.
Recall that~$\mathcal P_d(\lambda,\rho)$ denotes the set of homogeneous polynomials with~$\norm{P}_1\leq \lambda$ and~$\norm{P}_2\leq \rho$.
From Lemma~\ref{lem:proof_2_1} and the estimate~\eqref{eq:dd'-1} it follows that
\begin{align*}
	& \sup_{P \in \mathcal P_d(\lambda,\rho)} \E\abs{P(X(t+1))-P(x(t+1))} \\
	& \qquad\leq 2\rho + 2^d \sup_{P \in \mathcal P_d(\lambda,\rho)} \E\abs{P(X(t))-P(x(t))}
		+ 2^d M \lambda \sup_{P \in \mathcal P_d(\lambda,\rho)} \E\abs{P(X(t))-P(x(t))} \\
	& \qquad\leq 2\rho + 2^d (1+M\lambda) \sup_{P \in \mathcal P_d(\lambda,\rho)} \E\abs{P(X(t))-P(x(t))}.
\end{align*}
Therefore, by induction,
\[
	\sup_{P \in \mathcal P_d(\lambda,\rho)} \E\abs{P(X(t))-P(x(t))}
	\leq 2\rho \sum_{s=0}^{t-1} 2^{ds}(1+M\lambda)^s.
\]
A very crude upper bound leads to the statement of Theorem~\ref{thm:poly}.
\end{proof}

Theorem~\ref{thm:poly} gives a uniform bound
over an infinite set of polynomial observables.
For a specific polynomial observable
one might obtain a more detailed bound by
sharpening the deterministic part of the argument,
as done in Theorem~\ref{thm:main} for the neighborhood average as a
specific observable of degree $d=1$.
See also~Section~\ref{sec:discussion}.


\section{Application I: A Deterministic Method} \label{sec:dynamics}
\noindent
Approximating stochastic processes by deterministic processes has
a long history in both probability and
combinatorics; see e.g.,~\cite{wormald1999differential, wormald1995differential, bennett2023extending, duckworth2002minimum, benaim2003deterministic, sussmann1978gap}.
Note that here the stochastic process~$X(t)$ evolves on~$\{0,1\}^V$, while the deterministic process~$x(t)$ evolves in~$[0,1]$.
So~$x_v(t)$ is typically an imprecise approximation of~$X_v(t)$.
The idea of this section is that we use Theorem~\ref{thm:main} so as to obtain a \emph{deterministic method}, showing that~$x_v(t)$ is a `good' approximation of~$\E(X_v(t))$.

\begin{corollary} \label{cor:vertices}
Under the hypothesis of Theorem~\eqref{thm:main}, for every~$v\in V$ and~$t\geq 0$ we have
\[
	\abs{\E X_v(t) - x_v(t)} \leq \sum_{r=0}^{t-1}
		\sum_{s=0}^{r-1} {r \choose {s+1}}
		\pp{\sum_{w\in V} \frac{\pi_{v,w}(s)}{\sqrt{\deg(w)}}} M^s.
\]
In particular,
\begin{equation} \label{eq:vertices_limit}
	\lim_{\delta(G)\to\infty} \abs{\E X_v(t) - x_v(t)} = 0.
\end{equation}
\end{corollary}

The limit~\eqref{eq:vertices_limit} has to be understood as follows. Fix~$t\geq 0$.
For any graph sequence~$(G_n)_{n\geq 0}$ satisfying~$\delta(G_n)\to\infty$,
for any sequence of initial conditions, and for any vertex~$v$, the limit holds.
The limit is uniform in the set of graphs with given minimum degree,
in the initial condition and in the vertex~$v$. The limit is uniform over any bounded time
interval, but it is not uniform over all~$t\geq 0$.
Note that~$\delta(G_n)\to\infty$
implies~$\abs{V(G_n)}\to\infty$ with~$V(G_n)$ being the set of vertices of~$G_n$, but the converse does not hold.

\begin{proof} [Proof of Corollary~\ref{cor:vertices}]
From the law of total expectation we obtain
\begin{align*}
	& \E X_v(t+1) - x_v(t+1) \\
	& = \E\p{\E^{X(t)} X_v(t+1) - x_v(t+1)} \\
	& = (1-f(x_{N_v}(t)) - g(x_{N_v}(t))) (\E X_v(t)-x_v(t)) \\
		& \qquad + \E\pp{(f(x_{N_v}(t)) - f(X_{N_v}(t))) X_v(t)
		+ (g(X_{N_v}(t)) - g(x_{N_v}(t))) (1- X_v(t))}.
\end{align*}
Since~$0\leq f,g\leq 1$ and~$X_v(t)\in \{0,1\}$, the
triangle inequality and $M$-Lipschitz continuity imply
\[
	\abs{\E X_v(t+1) - x_v(t+1)} \leq \abs{\E X_v(t)-x_v(t)} + M \E \abs{X_{N_v}(t)-x_{N_v(t)}}.
\]
The statement follows now by induction from Theorem~\ref{thm:main}.
\end{proof}

This corollary implies that the deterministic process---an iterated map---gives information
about the evolution of the stochastic process.
Note that this is particularly useful in simulations. 
The deterministic process~$x(t)$ is obtained from a $\abs{V}$-dimensional map while~$\E X(t)$ is obtained from a $2^{\abs{V}}$-dimensional Markov 
chain, which is often computationally intractable~\cite{brand2015rapid}.
This reduces computational complexity from~$\O(\abs{V}^3)$ to~$\O(8^{\abs{V}})$.\footnote{Squaring a matrix of size~$n\times n$ costs~$\O(n^3)$.
Computing the $t$-th power by repeated squaring costs~$\O(\log t n^3)$.}

While convenient for simulations, for most purposes one cannot
simply replace~$\E X(t)$ by~$x(t)$; we will discuss this in more detail in Section~\ref{sec:discussion}.
First, results about long time behavior of~$x(t)$ do not apply to~$X(t)$,
since Corollary~\ref{cor:vertices} is a finite-time approximation.
For instance, in SIS models,
the disease-free state~$X=0$ can be absorbing, meaning that it is eventually
reached and never left in the stochastic dynamics, while in contrast
the corresponding state~$x=0$ can be unstable in the deterministic dynamics.
Second, one is interested in the specific finite set of initial
conditions~$\{0,1\}^V\subseteq [0,1]^V$ to understand the stochastic process~$X$ rather than results about global dynamics of~$\Gamma: [0,1]^V\to [0,1]^V$.

In the remainder of this section, we show how Corollary~\ref{cor:vertices} can be used
to approach the following problem: If~$(X_v(0))_{v\in V}$ are independent identically
distributed, how far are~$X_v(t)$ from being identically distributed?
This problem will be relevant in Section~\ref{sec:random_graphs} in the context
of dynamic random graphs with \er initial condition.

Let~$\ber(p)$ be a Bernoulli distributed random variable in which~$1$ has probability~$p$.
If~$(x_v(0))_{v\in V}$ are sampled independently from~$\ber(p)$, we denote~$x(0)\sim \ber(p)^V$.
Note that if~$x(0)\sim \ber(p)^V$ then~$x(t)$, obtained deterministically from~$x(0)$,
is a random variable. First, we state a weaker version of the result.

\begin{proposition} [Diagonal Concentration---Qualitative] \label{prop:diag_con_intro}
Suppose that~$1-f=g$
and, moreover, that~$\log \abs{V(G)} = o(\delta(G))$.
Fix~$\ve>0$, $p\in [0,1]$ and~$t\geq 1$.
Then as~$\abs{V(G)}\to\infty$, for asymptotically
almost every initial condition~$x(0) \sim \ber(p)^{V(G)}$
we have~$\abs{x_v(t) - g^t(p)}\leq \varepsilon$ for all~$v\in V$.
\end{proposition}

Assume that~$\abs{x_v(t) - g^t(p)}\leq \ve$ for all~$v\in V$, as in the
conclusion of the proposition. Then~$x(t)$
is in an $\ve$-neighborhood of the diagonal~$\Delta\subseteq [0,1]^V$, thus the name of the result.
Moreover, since~$x_v(t) = \E X_v(t) = \P (X_v(t)=1)$, Corollary~\ref{cor:vertices} implies
that~$(X_v(t))_{v\in V}$ are approximately identically distributed.

Proposition~\ref{prop:diag_con_intro} follows from
the following proposition.

\begin{proposition} [Diagonal Concentration---Quantitative] \label{prop:diag_con}
Let~$V=V(G)$ and~$\delta=\delta(G)$.
Fix~$p\in [0,1]$ and let~$(x_v(0))_{v\in V}$ be sampled independently from~$\ber(p)$.
Define~$\theta = \abs{1-f(p)-g(p)}$ and~$q=\max\{p,1-p\}$.
Then for every~$\ve>0$, with probability
at least~$1-2\abs{V} e^{-2 \delta \ve^2}$
\begin{equation} \label{eq:diag_con}
		d_\infty(x(t),\tilde\gamma^t(p) \1_V) \leq (q\theta + M\ve) e^{2M(t-1)}
		\quad \text{for all } t\geq 1,
\end{equation}
where $\tilde\gamma^t$ is the $t$-fold iterate of~$\tilde\gamma(s)=\gamma(s,s)=s(1-f(s)) + (1-s)g(s)$.
In particular, if~$\theta=0$ and~$\log \abs{V(G)} = o(\delta(G))$, then for every~$t\geq 1$
\begin{equation} \label{eq:diag_con_lim}
	\lim_{\abs{V(G)}\to\infty} d_\infty(x(t),\Delta) = 0 \quad \text{in probability}.
\end{equation}
\end{proposition}

The quantity~$\theta$ in Proposition~\ref{prop:diag_con} has an intuitive interpretation.
Note that~$1-f(X_{N_v}(t))$ is the probability
that~$X_v(t+1)=1$ given~$X_v(t)=1$, while~$g(X_{N_v}(t))$
is the probability that~$X_v(t+1)=1$ given~$X_v(t)=0$.
Therefore, if~$1-f=g$ then the probability that~$X_v(t+1)=1$ does not depend on~$X_v(t)$.
Note that in this case~$\tilde\gamma(s)=g(s)$.

The limit~\eqref{eq:diag_con_lim} has to be understood as follows:
For any graph sequence~$(G_n)_{n\geq 0}$
satisfying~$\abs{V(G_n)}\to\infty$
and~$\log \abs{V(G_n)} = o(\delta(G_n))$, and for any sequence of initial conditions
satisfying the hypothesis, the limit holds. If the initial conditions are chosen
independently for each~$n$, one can construct the sequence of random processes
on a common probability space and obtain almost sure convergence; however, such a construction 
seems rather unnatural.

\begin{proof} [Proof of Proposition~\ref{prop:diag_con}]
Fix a vertex~$v\in V$. Let~$h\in \{f,g\}$.
By Hoeffding's inequality and $M$-Lipschitz continuity of~$h$,
it follows that with probability at least~$1-2 e^{-2 \delta \ve^2}$ we have
\begin{equation} \label{eq:ineq}
	\abs{x_{N_v}(0) - p} \leq \ve, \quad \abs{h(x_{N_v}(0)) - h(p)} \leq M\ve.
\end{equation}
By the union bound, with probability at least~$1-2 \abs{V} e^{-2 \delta \ve^2}$
the inequalities~\eqref{eq:ineq} hold for all~$v\in V$.
Let~$q=\max\{p,1-p\}$.
Since~$x_v(0)\in\{0,1\}$ and~$p\in [0,1]$,
the inequalities~\eqref{eq:ineq} imply
\begin{align*}
	x_v(1) 	& = x_v(0) (1-f(x_{N_v}(0))) + (1-x_v(0)) g(x_{N_v}(0)) \\
			& \in \{x_v(0) (1-f(p)) + (1-x_v(0)) g(p)\} + (-M\ve,+M\ve) \\
			& \subseteq \{pf(p) + (1-p)g(p)\} + (-q \theta, +q\theta) + (-M\ve,+M\ve)
\end{align*}
for every~$v\in V$, thus~$d_\infty(x(1), \tilde\gamma(p) \1_V) \leq q \theta + M\ve$.
The map~$\Gamma:[0,1]^V \to [0,1]^V$
with components~$\Gamma(x)_v = x_v(1-f(x_{N_v})) + (1-x_v) g(x_{N_v})$
is~$2M+1$ Lipschitz and the diagonal~$\Delta$ is $\Gamma$-invariant.
It follows that, with the same probability, for every~$t\geq 1$ we have
\[
	d_\infty(x(t),\Delta) \leq (q \theta + M\ve) (2M+1)^{t-1}
		\leq (q \theta + M\ve) e^{2M(t-1)}.
\]
This completes the proof.
\end{proof}


\section{Application II: Dynamic Random Graphs} \label{sec:random_graphs}
\noindent
As a second application, we consider the evolution of dynamic random graphs, where edges of the complete graph~$K_n$ turn on and off.
More precisely, we view this as an occupancy process on~$G=L(K_n)$, the line graph of~$K_n$, with vertices~$V(G)=E(K_n)$.
Labeling vertices of the underlying graph~$K_n$ as~$V(K_n)=\{1,\dotsc,n\}=[n]$, every vertex of the line graph~$v\in V(G)$ is of the form~$v=\{j,k\}$ for distinct~$j,k\in [n]$.
The stochastic process~$X(t):V(G)\to \{0,1\}$ now describes the state of each edge, which can be interpreted as a sequence of random graphs determined by the edges of~$K_n$ (vertices of~$G$) with state~$1$.
To avoid confusion, we label the coordinates of the stochastic process as~$(X_e(t))_{e\in E(K_n)}$ rather than~$(X_v(t))_{v\in V(G)}$.

If~$1-f(x)=g(x)=p$ for every~$x\in [0,1]$, then~$X(t)$ is a random graph
with edges being active with probability~$p$, independently of each other and of the past.
If~$f(x)=g(x)=p$,
then edges turn on and off with probability~$p$ independently of each
other but not on the past
(a continuous-time analogue of this process is considered in~\cite{braunsteins2023sample}).
For general~$f,g$ edges depend on each other and on the past.
We motivate the results of this section through the following
running example.

\begin{example} [\er initial condition] \label{ex:er_1}
Fix~$p\in [0,1]$ and
let~$(X_e(0))_{e\in E(K_n)}$ be sampled independently from~$\ber(p)$,
then~$X(0)=\G(n,p)$ is an \er random graph with edge probability~$p$.
Moreover, let~$1-f(x)=g(x)$ for every~$x\in [0,1]$.

We claim that for fixed~$T\geq 0$ and for~$n$ large~$(X(t))_{0\leq t\leq T}$ approximates
the sequence of \er random graphs with time-dependent
edge density~$(\G(n, g^t(p)))_{0\leq t\leq T}$, where $g^t$~is the $t$-fold interate of~$g$. Let~$x(0)=X(0)$, then~$x(0)\sim \ber(p)^{E(K_n)}$.
From Corollary~\ref{cor:vertices} it follows that~$\E(X_e(t)) \approx x_e(t)$ for~$n$ large.
Since~$\abs{V(G)} = \abs{E(K_n)} = {n\choose 2}$ and~$\delta(G) = 2n-4$,
we have~$\log \abs{V(G)} = o(\delta(G))$, thus Proposition~\ref{prop:diag_con_intro} applies,
showing that~$x_e(t)\approx g^t(p)$ for all~$e\in E(K_n)$ for~$n$ large.
This means that, informally,~$(x_e(t))_{e\in E(K_n)}$ are approximatively identically distributed~$\ber(g^t(p))$
for~$n$ large. Interpreting~$v\in V$ as edges in~$K_n$,
this shows that edge density in~$(X(t))_{0\leq t\leq T}$
approximates edge density in~$(\G(n, g^t(p)))_{0\leq t\leq T}$.
\end{example}

The next goal is to extend the observations made in Example~\ref{ex:er_1}
about the edge density to other homomorphism densities (i.e., subgraph densities).

\begin{definition}
Let~$F$ and~$G$ be finite graphs. We call \emph{homomorphism} a map~$\varphi: V(F)\to V(G)$
such that~$\{\varphi(k),\varphi(j)\} \in E(G)$ for every~$\{k,j\}\in E(F)$.
We call \emph{homomorphism density}
the probability that a map~$V(F)\to V(G)$ chosen uniformly at random
is an homomorphism. This probability is denoted by~$\t(F,G)$.
\end{definition}

Fix~$F=(V(F),E(F))$.
We identify~$X(t)$ with the graph
with vertices~$[n]=\{1,\ldots,n\}$ and edges the pairs~$\{k,j\}$ such that~$X_{\{k,j\}}(t)=1$.
By abuse of notation,
we denote by~$\t(F,X(t))$ the homomorphism density of~$F$ with respect to this graph.
If we extend~$X$ to trivial pairs as~$X_{\{k,k\}}=0$, then
we can write
\begin{align*}
	\t(F,X(t))
	& = \frac{1}{n^{\abs{V(F)}}} \sum_{k_1,\ldots,k_{\abs{V(F)}} \in [n]}
		\prod_{\{\alpha,\beta\}\in E(F)} X_{\{k_\alpha,k_\beta\}}(t) \\
	& = \p{\sum_{W \in {E(K_n) \choose \abs{E(F)}}} P_W \prod_{e\in W} X_e(t)} + o(1)
		\quad \text{as}\ n\to\infty
\end{align*}
where~$P_W$ is equal to~$\aut(F) n^{-\abs{V(F)}}$
if the set of edges~$W\subseteq E(K_n)$ induces a graph
isomorphic to~$F$ and is equal to~$0$ otherwise, and
where~$o(n)$ includes all non-injective choices of~$k_1,\ldots,k_{\abs{V(F)}}$.
Therefore, up to an absolute error that vanishes as~$n\to\infty$,
the homomorphism density~$\t(F,X(t))$ is equal to an $\abs{E(F)}$-homogeneous polynomial~$P$
as in Definition~\ref{def:combinatorial_norms}, and Theorem~\ref{thm:poly} applies.

\begin{corollary} [Concentration of Homomorphism Densities] \label{cor:hom}
Suppose that~$f$ and~$g$ are $M$-Lipschitz. Let~$G=L(K_n)$.
Fix~$x(0)\in \{0,1\}^{E(K_n)}$ and let~$\P(X(0)=x(0))=1$. Then
\[
	\E\abs{\t(F,X(t)) - \t(F,x(t))} = \O_{M,F,t}(n^{-1}).
\]
\end{corollary}
\begin{proof}
We prove the claim by analyzing~$\norm{P}_1$ and~$\norm{P}_2$.
Note that once~$e\in E(K_n)$ is fixed, there are at most~$n^{\abs{V(F)}-2}$
ways of choosing~$W\in {E(K_n) \choose \abs{E(F)}}$ with~$e\in W$
such that~$W$ induces a graph isomorphic to~$F$. Therefore~$c_e(P) = \O(n^{-2})$.
Since there are~${n\choose 2}$ many ways
of choosing~$e\in E(K_n)$, we obtain
\begin{align*}
	& \norm{P}_1 = \sum_{W: e\in W \in {E(K_n) \choose \abs{E(F)}}} \abs{P_W} = \mathcal O(1),
	& \norm{P}_2 = \p{\sum_{W: e\in W \in {E(K_n) \choose \abs{E(F)}}} \abs{P_W}^2}^{\frac{1}{2}}
		= \mathcal O(n^{-1}).
\end{align*}
The statement now follows from Theorem~\ref{thm:poly}.
\end{proof}

\begin{example} [\er initial condition, continued] \label{ex:er_2}
In Example~\ref{ex:er_1} we showed that for fixed~$t\geq 0$ and~$n$ large
we have~$x_e(t)\approx g^t(p)$ for all~$e\in E(K_n)$,
thus in particular
\[
	\t(F,x(t))\approx g^t(p)^{\abs{E(F)}} = \t(F,\G(n,g^t(p)))
\]
and therefore from Corollary~\ref{cor:hom} it follows that
\[
	\E\abs{\t(F,X(t)) - \t(F,\G(n,g^t(p)))} = o(1) \quad \text{ as } n\to\infty.
\]
This shows that, informally, for~$n$ large with high probability a graph sampled
from the distribution~$X(t)$ looks like a graph sampled from the
distribution~$\G(n,g^t(p))$. In other words,
if for each~$n$ a different graph is sampled independently,
the resulting graph sequence is quasi-random~\cite{chung1989quasi, lovasz2012large}.
\end{example}

\begin{example} [\er initial condition with specific coupling] \label{ex:er_3}
As a very concrete instance, take
take~$1-f(x)=g(x)=4x(1-x)$ for every~$x\in [0,1]$.
As first shown by Ulam and von Neumann, the map~$x\mapsto 4x(1-x)$
is ergodic and mixing on the unit interval~\cite{ulam1947, jafarizadeh2001hierarchy}.
In particular, for Lebesgue-almost every~$p\in [0,1]$ the trajectory~$g^t(p)$ is Lebesgue-dense
and is highly sensitive to the choice of~$p$.
In this case, as~$n\to\infty$, the random graph~$X(t)$ approximates a sequence of \er
random graphs with time-dependent edge density evolving chaotically over time.
\end{example}


\section{Discussion and Concluding Remarks} \label{sec:discussion}
\noindent
To conclude, we discuss our results and possible extensions in a broader context---in particular with a view on dynamical systems---and identify some directions for further research.

\subsection{Generalizations and extensions}
Theorem~\ref{thm:main} and Theorem~\ref{thm:poly} bound
deviations in polynomial observables for arbitrary finite graphs.
Importantly, these bounds do not require graphs to be sparse nor dense.
For neighborhood averages,
these deviation bounds have the form~$\E(\abs{P(X)-\E(P(X))})\leq \ve$,
with~$\ve$ decreasing polynomially in the graph growth.
As we have seen in Proposition~\ref{prop:diag_con},
if graphs are required to be dense,
e.g., $\log \abs{V(G)} = o(\delta(G))$,
then our techniques give stronger bounds, namely exponential
bounds for~$\P(\abs{P(X)-\E(P(X))})$.

In principle, the asymptotic bound obtained in Corollary~\ref{cor:hom} for homomorphism densities can be replaced by a combinatorial bound, extracted from the proof of Theorem~\ref{thm:poly} in the same way in which the combinatorial bound for neighborhood averages is extracted from the proof of Theorem~\ref{thm:main}.

While we discussed vertex processes and edge processes separately, our results are directly applicable to vertex-edge processes in which vertices and edges states interact.
This is a common setup in adaptive or co-evolutionary network dynamics~\cite{BERNER20231} such as SIS dynamics on an evolving network. 
Similar to how edge processes can be described as vertex processes on the line graph, vertex-edge processes can be describes as vertex processes on the following graph:
Consider the disjoint union of the graph~$G$
and the line graph~$L(G)$ and connect~$v\in V(G)$ with~$e\in V(L(G))$ if~$v\in e$.
At first, this approach would treat vertices and edges equally, which is undesirable, but a distinction can be made
by introducing different~$f,g$ for vertex-vertex, edge-edge, and vertex-edge interactions.
Our techniques extend with little modifications.
Similar extensions are possible for processes on and of hypergraphs; see also~\cite{doi:10.1137/21M1414024}.

The dynamic random graphs considered in this paper are Markov chains in the space of all graphs with vertices~$\{1,\ldots,n\}$.
One might want to restrict this space by preventing some
pairs of vertices to become edges.
This can be done by replacing the
line graph of the complete graph~$L(K_n)$
with the line graph of some other graph~$L(H)$.
In the case~$L(K_n)$ we used homorphism densities
to compare the resulting random graph with an \er random graph, see Example~\ref{ex:er_2}). If~$H$ is particularly sparse,
it might be more insightful to normalize homorphism densities,
or to use other metrics altoghether, such as
local convergence~\cite{hatami2014limits, van2021giant}.
An alternative approach would be to work with graph limits and obtain
finite results as approximations. However, in the context of line graphs,
this would require developing large deviations
in the space of \textit{graphops}~\cite{Backhausz_Szegedy_2022} rather than \textit{graphons}.

\subsection{Stochastic processes and dynamical systems}
We conclude with some final remarks on deterministic methods, as a techniques to approximate stochastic processes, from a dynamical systems perspective.
The deterministic method that arises from our main results (cf.~Section~\ref{sec:dynamics}) considers---as do other deterministic methods such as~\cite{allen2000comparison}---dynamics for fixed times $t\leq T$ as the number of vertices goes to infinity.
By contrast, standard questions in dynamical systems theory often concern the long term dynamics $t\to\infty$ as system properties (e.g., the underlying graph~$G$ if it has a network structure) are kept constant.
While deterministic methods are extremely useful (e.g., due to the decreased computational complexity as discussed in Section~\ref{sec:dynamics}), applying dynamical systems techniques to~\eqref{eq:det_process} to understand the stochastic evolution~\eqref{eq:stoch_process} comes with pitfalls as we illustrate with two concrete examples. Both examples are natural and
show how the small-$t$
large-$G$ regime can exhibit significantly different behavior
than the large-$t$ small-$G$ regime.

The first concrete example, already discussed in Section~\ref{sec:dynamics},
is the SIS model on finite graphs.
The healthy state~$X=0$ can be absorbing, meaning that it is eventually
reached and never left in the stochastic dynamics, while in contrast
the corresponding state~$x=0$ can be unstable in the deterministic dynamics.

The second concrete example concerns dynamic random graphs~$X(t)$
with \er initial conditions~$\G(n,p)$.
As explained in Example~\ref{ex:er_2}, Proposition~\ref{prop:diag_con}
shows that edges in~$X(t)$ are approximately identically distributed
for~$0<t<\O(\log(n))$: For each~$t>0$ the vector~$x(t)$ converges the diagonal of~$[0,1]^E$ as~$n\to\infty$. In contrast, assuming
regularity in~$f$ and~$g$ and linearizing the deterministic system~$x(t)$
reveals that the diagonal can be linearly unstable.
In other words, the limits~$\lim_{n\to\infty}$
and~$\lim_{t\to\infty}$ are not interchangeable.


\bibliographystyle{unsrt}
\bibliography{refs}

\end{document}